\documentclass[12pt,reqno]{amsart}
\usepackage[english]{babel}
\usepackage{amsmath,amsfonts,amssymb,amsthm,amscd,latexsym}
\usepackage[T1]{fontenc}
\usepackage[utf8]{inputenc}

\usepackage{tikz}
\usetikzlibrary{arrows}
\usepackage{color}
\usepackage{cite}
\usepackage{multirow}
\usepackage{cite}

\newtheorem{theorem}{Theorem}[section]
\newtheorem{lemma}[theorem]{Lemma}
\newtheorem{proposition}[theorem]{Proposition}

\theoremstyle{definition}
\newtheorem{definition}[theorem]{Definition}

\theoremstyle{remark}
\newtheorem{remark}[theorem]{Remark}
\numberwithin{equation}{section}

\begin{document}

\setcounter{page}{1}

\title[Local and 2-local derivations of  solvable Leibniz algebras]{Local and 2-local derivations of solvable Leibniz algebras}

\author[ Ayupov Sh.A., Khudoyberdiyev A.Kh., Yusupov B.B. ]{Shavkat Ayupov$^{1,2}$, Abror Khudoyberdiyev$^{1,2}$, Bakhtiyor Yusupov$^2$}
\address{$^1$ V.I.Romanovskiy Institute of Mathematics\\
  Uzbekistan Academy of Sciences, 81 \\ Mirzo Ulughbek street, 100170  \\
  Tashkent,   Uzbekistan}
\address{$^2$ National University of Uzbekistan, 4, University street, 100174, Tashkent, Uzbekistan}
\email{\textcolor[rgb]{0.00,0.00,0.84}{sh$_{-}$ayupov@mail.ru}}

\email{\textcolor[rgb]{0.00,0.00,0.84}{khabror@mail.ru}}
\email{\textcolor[rgb]{0.00,0.00,0.84}{baxtiyor\_yusupov\_93@mail.ru}}
\maketitle

\begin{abstract} We show that any local derivation on the solvable Leibniz algebras with model or abelian nilradicals,
whose the dimension of complementary space is maximal is a derivation. We show that solvable Leibniz algebras with abelian nilradicals,
which have $1$-dimension complementary space, admit local derivations which are not derivations. Moreover, similar
problem concerning $2$-local derivations of such algebras are investigated and an example of solvable Leibniz algebra given such that any $2$-local derivation on it is a derivation, but which admit local derivations which are not derivations.

\end{abstract}
\\


\section{Introduction}

In recent years non-associative analogues of classical constructions become of
interest in connection with their applications in many branches of mathematics and
physics. The notions of local and $2$-local derivations are also become popular for some non-associative algebras such as Lie and Leibniz algebras.

The notions of local derivations were introduced in 1990 by R.V.Kadison \cite{Kadison} and D.R.Larson, A.R.Sourour \cite{Larson}. Later in 1997, P.\v{S}emrl introduced the notion of  $2$-local derivations and
$2$-local automorphisms on algebras \cite{Sem}. The main problems concerning these notions are to find conditions under which all local ($2$-local) derivations become (global)derivations and to present examples of algebras with local ($2$-local) derivations that
are not derivations.

Investigation of local derivations on Lie algebras was initiated in papers in \cite{Ayupov7} and \cite{ChenWang} . Sh.A.Ayupov and K.K.Kudaybergenov
have proved that every local
derivation on semi-simple Lie algebras is a derivation and gave examples of nilpotent
finite-dimensional Lie algebras with local derivations which are not derivations. In \cite{Ayupov6} local derivations of solvable Lie algebras are investigated and it is shown that in the class of solvable Lie algebras there exist algebras which admit
local derivations which are not ordinary derivation and also algebras for which every local derivation
is a derivation. Moreover, it is proved that every local derivation on a finite-dimensional solvable Lie algebra with model
nilradical and maximal dimension of complementary space is a derivation.
In \cite{AKB} local derivations and automorphism of complex finite-dimensional simple Leibniz algebras are investigated. They proved that all local derivations on a finite-dimensional complex simple Leibniz algebra are automatically derivations and it is shown that filiform Leibniz algebras admit local derivations which are not derivations.

Several papers have been devoted to similar notions and corresponding problems for $2$-local derivations and automorphisms of finite-dimensional Lie and Leibniz algebras \cite{AyuKud, AKB, AyuKudRak, ChenWang}.
Namely, in
\cite{AyuKudRak} it is proved that every 2-local derivation on a
semi-simple Lie algebra is a derivation and that
each finite-dimensional nilpotent Lie algebra, with dimension
larger than two admits 2-local derivation which is not a
derivation. Concerning 2-local
automorphism, Z.Chen and D.Wang in \cite{ChenWang} prove that if \(\mathcal{L},\) is
a simple Lie algebra of type $A_{l},D_{l}$ or $E_{k}, (k = 6, 7,
8)$ over an algebraically closed field of characteristic zero,
then every 2-local automorphism of \(\mathcal{L},\) is an automorphism. Finally,
in \cite{AyuKud} Sh.A.Ayupov and K.K.Kudaybergenov generalized this result
of \cite{ChenWang} and proved that every 2-local automorphism of a
finite-dimensional semi-simple Lie algebra over an algebraically
closed field of characteristic zero is an automorphism. Moreover,
they also showed that every nilpotent Lie algebra with finite
dimension larger than two admits $2$-local automorphisms which is
not an automorphism.

 In \cite{Ayupov8}, \cite{Yusupov1} Sh.A.Ayupov and B.B.Yusupov investigated  $2$-local derivations on infinite-dimensional Lie algebras over a field of characteristic zero. They proved that all $2$-local
derivations on the Witt algebra as well as on the positive Witt algebra are (global)
derivations, and give an example of infinite-dimensional Lie algebra with a 2-local
derivation which is not a derivation.


In the present paper we study local and $2$-local derivations of solvable Leibniz algebras.
We show that any local derivation on solvable Leibniz algebras with model or abelian nilradicals, whose the dimension of complementary space is maximal is a derivation, but solvable Leibniz algebras with abelian nilradical,
whose have $1$-dimension complementary space admit local derivations which are not derivations. Moreover, similar
problems concerning $2$-local derivations of such algebras are investigated.
\medskip

\section{Preliminaries}

In this section we give some necessary definitions and preliminary results.

\begin{definition}
A  vector space with bilinear bracket $(\mathcal{L},[\cdot,\cdot])$
is called a Leibniz algebra if for any $x,y,z\in L$ the so-called
Leibniz identity
$$\big[x,[y,z]\big]=\big[[x,y],z\big]-\big[[x,z],y\big],$$
holds.
\end{definition}

Here, we adopt the right Leibniz identity; since the bracket is
not skew-symmetric, there exists the version corresponding to the
left Leibniz identity,
 \[\big[[x,y],z\big] =  \big[x,[y,z]\big] - \big[y,[x,z]\big] \,. \]

Let $\mathcal{L}$ be a Leibniz algebra. For a Leibniz algebra $\mathcal{L}$ consider the following central lower and
derived sequences:
$$
\mathcal{L}^1=\mathcal{L},\quad \mathcal{L}^{k+1}=[\mathcal{L}^k,\mathcal{L}^1], \quad k \geq 1,
$$
$$\mathcal{L}^{[1]} = \mathcal{L}, \quad \mathcal{L}^{[s+1]} = [\mathcal{L}^{[s]}, \mathcal{L}^{[s]}], \quad s \geq 1.$$

\begin{definition} A Leibniz algebra $\mathcal{L}$ is called
nilpotent (respectively, solvable), if there exists  $p\in\mathbb N$ $(q\in
\mathbb N)$ such that $\mathcal{L}^p=0$ (respectively, $\mathcal{L}^{[q]}=0$).The minimal number $p$ (respectively, $q$) with such
property is said to be the index of nilpotency (respectively, of solvability) of the algebra $\mathcal{L}$.
\end{definition}
Note that any Leibniz algebra $\mathcal{L}$ contains a unique maximal solvable (resp. nilpotent) ideal, called the radical (resp. nilradical) of the algebra.

A derivation on a Leibniz algebra $\mathcal{L}$ is a linear map  $D: \mathcal{L} \rightarrow \mathcal{L}$  which satisfies the Leibniz rule:
\begin{equation}\label{der}
D([x,y]) = [D(x), y] + [x, D(y)] \quad \text{for any} \quad x,y \in \mathcal{L}.
\end{equation}

The set of all derivations of a Leibniz algebra $\mathcal{L}$ is a Lie algebra with respect to
commutation operation and it is denoted by $Der(\mathcal{L}).$

For any element $x\in \mathcal{L},$ the operator of
right multiplication $ad_x: \mathcal{L} \to \mathcal{L}$, defined as $ad_x(z)=[z,x]$ is a derivation, and derivations of this form are called inner derivation. The set of all inner derivations of $\mathcal{L},$ denoted by $ad (\mathcal{L}),$ is an ideal in $Der(\mathcal{L}).$

For a finite-dimensional nilpotent Leibniz algebra  $N$ and for the matrix of the linear operator $ad_x$ denote by $C(x)$ the
descending sequence of its Jordan blocks' dimensions. Consider the
lexicographical order on the set $C(N)=\{C(x)  \ | \ x \in N\}$.

\begin{definition} \label{d4} The sequence
$$\left(\max\limits_{x\in N\setminus N^2} C(x) \right) $$ is said to be the
characteristic sequence of the nilpotent Leibniz algebra $N.$
\end{definition}

\begin{definition}
A linear operator $\Delta$ is called a local derivation, if for any $x \in \mathcal{L},$ there exists a derivation $D_x: \mathcal{L} \rightarrow \mathcal{L}$ (depending on $x$) such that
$\Delta(x) = D_x(x).$ The set of all local derivations on $\mathcal{L}$ we denote by $LocDer(\mathcal{L}).$
\end{definition}

\begin{definition}
A map $\nabla :  \mathcal{L} \rightarrow \mathcal{L}$ (not
necessary  linear) is called a \emph{$2$-local derivation}, if for any $x,y\in \mathcal{L}$, there exists a derivation $D_{x,y}\in Der
(\mathcal{L})$ such that
{\small\[
\nabla(x)=D_{x,y}(x), \quad \nabla(y)=D_{x,y}(y).
\]}
\end{definition}

\subsection{Solvable Leibniz algebras with abelian nilradical}
Let $\textbf{a}_n$ be the $n$-dimensional abelian algebra and let $R$ be a solvable algebra with nilradical $\textbf{a}_n$.
Take a basis $\{f_1,f_2, \dots, f_n, x_1, x_2, \dots x_k\}$ of $R,$ such that $\{f_1,f_2, \dots, f_n\}$ is a basis of $\textbf{a}_n.$
In \cite{Adashev} such solvable algebras in case of $k=n$ are classified and it is proved that
any $2n$-dimensional solvable Leibniz algebra with nilradical $\textbf{a}_n$ is isomorphic the direct sum of two dimensional algebras, i.e., isomorphic to the algebra
  $${\mathcal L_t}: [f_j,x_j] = f_j, \quad [x_j,f_j]=\alpha_jf_j, \quad 1\leq j \leq n,$$
where $\alpha_j\in\{-1,0\}$ and $t$ is the number of zero  parameters $\alpha_j.$

Moreover, in the following theorem the classification of $(n+1)$-dimensional solvable Leibniz algebras with $n$-dimensional abelian nilradical is given.


\begin{theorem}\label{thmQ=1}\cite{Adashev} Let $R$ be a $(n+1)$-dimensional solvable Leibniz algebra with $n$-dimensional abelian nilradical. If $R$ has a basis $\{f_1,f_2, \dots, f_n, x\}$ such that the operator $ad_x|_{\textbf{a}_n}$ has Jordan block form, then it is  isomorphic to one of the following two non-isomorphic algebras:
\[R_1: \left\{\begin{array}{ll}
[f_i,x]=f_i+f_{i+1},&1\leq i\leq n-1,\\[1mm]
[f_n,x]=f_n,&\\[1mm]
\end{array}\right.
R_2: \left\{\begin{array}{ll}
[f_i,x]=f_i+f_{i+1},&1\leq i\leq n-1,\\[1mm]
[f_n,x]=f_n,&\\[1mm]
[x,f_i]=-f_i-f_{i+1},&1\leq i\leq n-1,\\[1mm]
[x,f_n]=-f_n.&\\[1mm]
\end{array}\right.\]
\end{theorem}

In the following propositions, we present the general form of a derivation of the algebras ${\mathcal L_t},$ $R_1$ and $R_2$.

\begin{proposition}\cite{Adashev}  Any derivation $D$ of the algebra ${\mathcal L_t}$ has the following form:
\[ \quad D(f_j)=a_jf_j, \quad D(x_j)=\alpha_jb_jf_j, \quad 1\leq j\leq n,\]
where $\alpha_j\in\{-1,0\}$ and $t$ is the number of zero  parameters $\alpha_j.$
\end{proposition}

\begin{proposition}\label{123} Any derivation $D$ of the algebras $R_1$ and $R_2$ has the following form:
\[Der(R_1):
D(f_i)=\sum\limits_{j=i+1}^n\alpha_{j-i+1}f_{j}+\alpha_1f_i.
\]
\[Der(R_2):\left\{\begin{array}{lll}
D(f_i)&=\sum\limits_{j=i+1}^n\alpha_{j-i+1}f_{j}+\alpha_1f_i,\ \ \ \ 1\leq i\leq n,\\
D(x)&=\sum\limits_{j=1}^n\beta_jf_j.
\end{array}\right.
\]
\end{proposition}

%

\subsection{Solvable Leibniz algebras model nilradical}
Let $N$ be a nilpotent Leibniz algebra with the characteristic sequence
$(m_1,\ldots,m_s),$ and multiplication table
$$N_{m_1,\dots,m_s}: [e^t_i,e^1_1]=e^t_{i + 1},\ 1 \leq t\leq s,\, 1\leq i\leq m_t-1.$$

The algebra $N_{m_1,\dots,m_s}$ usually is said to be model Leibniz algebra.

\begin{theorem}\cite{Omirov} An solvable Leibniz algebra $R$ with nilradical $N_{m_1,\dots, m_s},$ such that $DimR - DimN_{m_1,\dots, m_s} =s$, is isomorphic to the  algebra:
 $$R(N_{m_1,\dots,m_s},s): \left\{\begin{array}{ll}
[e^t_i,e^1_1]=e^t_{i+1}, & 1\leq t\leq s, \,1\leq i\leq m_t-1 ,\\[1mm]
[e^1_i,x_1]=ie^1_i, & 1\leq i\leq m_1,\\[1mm]
[e^t_i,x_1]=(i-1)e^t_i, & 2\leq t\leq s, \,2\leq i\leq m_t,\\[1mm]
[e^t_i,x_t]=e^t_i, & 2\leq t\leq s, \,1\leq i\leq m_t,\\[1mm]
[x_1,e^1_1]=-e^1_1, \\[1mm]
 \end{array}\right.$$
where $\{x_1,\ldots x_s\}$ is a basis of complementary vector space.
\end{theorem}
\begin{proposition}\cite{Omirov}\label{pro}
Any derivation $D$ of the algebra $Der(R(N_{m_1,\dots, m_s},s))$ has the following form:
$$\begin{array}{ll}
D(e^1_i)=i\alpha_{1}e^1_i+\alpha_{2}e^1_{i+1},&1 \leq i \leq m_1-1,\\[2mm]
D(e^1_{m_1})=m_1\alpha_{1}e^1_{m_1},&\\[2mm]
D(e^t_i)=((i-1)\alpha_{1}+\beta_{t})e^t_i+\alpha_{2}e^t_{i+1},&  2\leq t\leq s ,\quad 1 \leq i \leq m_t-1,\\[1mm]
D(e^t_{m_t})=((m_t-1)\alpha_{1}+\beta_{t})e^t_{m_t},&  2\leq t\leq s,\\[1mm]
D(x_1)=-\alpha_{2}e^1_1.
 \end{array}$$
\end{proposition}
\begin{remark} Any derivation on the solvable Leibniz algebra $R(N_{m_1,\dots, m_s},s)$  is an inner derivation.
\end{remark}

\section{Local derivation of solvable Leibniz algebras}

\subsection{Local derivation of solvable  Leibniz algebras with nilradicals $R(N_{m_1,\dots, m_s},s)$}

Now we shall give the main result concerning local derivations of solvable  Leibniz algebra $R(N_{m_1,\dots, m_s},s).$
\begin{theorem}\label{thm11}
 Any local derivation on the solvable Leibniz algebra $R(N_{m_1,\dots, m_s},s)$  is a derivation.
\end{theorem}

\begin{proof} Let  $\Delta$ be a local derivation on $R(N_{m_1,\dots, m_s},s),$ then we have
$$
\Delta(x_i)=\sum\limits_{j=1}^sa_{i,j}x_j+\sum\limits_{p=1}^s\sum\limits_{j=1}^{m_p}b_{i,j}^pe_j^p,\quad
\Delta(e_{i}^t)=\sum\limits_{j=1}^sc_{i,j}^tx_j+\sum\limits_{p=1}^s\sum\limits_{j=1}^{m_p}d_{i,j}^{t,p}e_j^p.
$$

Let $D$ be a derivation on $R(N_{m_1,\dots, m_s},s),$ then by Proposition \ref{pro}, we obtain
$$\begin{array}{ll}
D(e^1_i)=i\alpha_ie^1_i+\beta_ie^1_{i+1},&1 \leq i \leq m_1-1,\\[2mm]
D(e^1_{m_1})=m_1\alpha_{m_1}e^1_{m_1},&\\[2mm]
D(e^t_1)=\sigma_te^t_1+\theta_te^t_2,& 2\leq t\leq s,\\[1mm]
D(e^t_i)=((i-1)\lambda_{i,t}+\mu_{i,t})e^t_i+\delta_{i,t}e^t_{i+1},&  2\leq t\leq s ,\quad 2 \leq i \leq m_t-1,\\[1mm]
D(e^t_{m_t})=((m_t-1)\xi_{i,t}+\eta_{i,t})e^t_{m_t},&  2\leq t\leq s,\\[1mm]
D(x_1)=-\gamma e^1_1.
 \end{array}$$

Considering the equalities $$\Delta(x_j)=D_{x_j}(x_j),\quad 1\leq j\leq s,$$  $$\Delta(e_i^t)=D_{e_i^t}(e_i^t),\quad 1\leq t\leq s,  \ 1\leq i\leq m_t,$$ we have
\[\left\{\begin{array}{lll}
\sum\limits_{j=1}^sc_{i,j}^1x_j+\sum\limits_{p=1}^s\sum\limits_{j=1}^{m_p}d_{i,j}^{1,p}e_j^p=i\alpha_ie^1_i+\beta_ie^1_{i+1},&  1 \leq i \leq m_1-1\\
\sum\limits_{j=1}^sc_{m_1,j}^1x_j+\sum\limits_{p=1}^s\sum\limits_{j=1}^{m_p}d_{m_1,j}^{1,p}e_j^p=m_1\alpha_{m_1}e^1_{m_1},\\
\sum\limits_{j=1}^sc_{1,j}^tx_j+\sum\limits_{p=1}^s\sum\limits_{j=1}^{m_p}d_{1,j}^{t,p}e_j^p=\sigma_te^t_1+\theta_te^t_2, & 2\leq t\leq s\\
\sum\limits_{j=1}^sc_{i,j}^tx_j+\sum\limits_{p=1}^s\sum\limits_{j=1}^{m_p}d_{i,j}^{t,p}e_j^p=((i-1)\lambda_{i,t}+\mu_{i,t})e^t_i+
\delta_{i,t}e^t_{i+1},&   2\leq t\leq s , \ 2 \leq i \leq m_t-1,\\
\sum\limits_{j=1}^sc_{m_p,j}^tx_j+\sum\limits_{p=1}^s\sum\limits_{j=1}^{m_p}d_{m_p,j}^{t,p}e_j^p=((m_t-1)\xi_{m_p,t}+\eta_{m_p,t})e^t_{m_t},&    2\leq t\leq s,\\
\sum\limits_{j=1}^sa_{1,j}x_j+\sum\limits_{p=1}^s\sum\limits_{j=1}^{m_p}b_{1,j}^pe_j^p=-\gamma e^1_1,\\
\sum\limits_{j=1}^sa_{1,j}x_j+\sum\limits_{p=1}^s\sum\limits_{j=1}^{m_p}b_{1,j}^pe_j^p=0,& 2\leq i\leq n.
\end{array}\right.
\]

Form the previous restrictions, we get that
$$\begin{array}{ll}
\Delta(e^1_i)=d_{i,i}^{1,1}e^1_i+d_{i,i+1}^{1,1}e^1_{i+1},&1 \leq i \leq m_1-1,\\[2mm]
\Delta(e^1_{m_1})=d_{m_1,m_1}^{1,1}e^1_{m_1},&\\[2mm]
\Delta(e^t_1)=d_{1,1}^{t,t}e^t_1+d_{1,2}^{t,t}e^t_2,& 2\leq t\leq s,\\[1mm]
\Delta(e^t_i)=d_{i,i}^{t,t}e^t_i+d_{i,i+1}^{t,t}e^t_{i+1},&  2\leq t\leq s ,\quad 2 \leq i \leq m_t-1,\\[1mm]
\Delta(e^t_{m_t})=d_{m_t,m_t}^{t,t}e^t_{m_t},&  2\leq t\leq s,\\[1mm]
\Delta(x_1)=b_{1,1}^1e^1_1.
 \end{array}$$

Consider
$$\Delta(e_1^1+e_1^t)=d_{1,1}^{1,1}e_1^1+d_{1,2}^{1,1}e_2^1+d_{1,1}^{t,t}e_1^t+d_{1,2}^{t,t}e_2^t.$$

On the other hand,
\begin{equation*}
\begin{split}
\Delta(e_1^1+e_1^t)&=D_{e_1^1+e_1^t}(e_1^1+e_1^t)=D_{e_1^1+e_1^t}(e_1^1)+D_{e_1^1+e_1^t}(e_1^t)=\\
                    & =\alpha_{e_1^1+e_1^t}e_1^1+\beta_{e_1^1+e_1^t}e_2^1+\eta_{e_1^1+e_1^t}e_1^t+\beta_{e_1^1+e_1^t}e_2^t.
\end{split}
\end{equation*}

Comparing the coefficients at the basis elements $e_2^1$ and $e_2^t,$ we get $\beta_{e_1^1+e_1^t}=d_{1,2}^{1,1},$ $\beta_{e_1^1+e_1^t}=d_{1,2}^{t,t},$
which implies $$d_{1,2}^{t,t}=d_{1,2}^{1,1}, \quad 2\leq t\leq s.$$

Similarly, considering $\Delta(e_1^1+e_i^1)$ for $3\leq i\leq m_1-1,$ we have
\begin{equation*}
\begin{split}
\Delta(e_1^1+e_i^1)&=d_{1,1}^{1,1}e_1^1+d_{1,2}^{1,1}e_2^1+d_{i,i}^{1,1}e_i^1+d_{i,i+1}^{1,1}e_{i+1}^1\\
 &=D_{e_1^1+e_i^1}(e_1^1+e_i^1)=D_{e_1^1+e_i^1}(e_1^1)+D_{e_1^1+e_i^1}(e_i^1)\\
                    & =\alpha_{e_1^1+e_i^1}e_1^1+\beta_{e_1^1+e_i^1}e_2^1+i\alpha_{e_1^1+e_i^1}e_i^1+\beta_{e_1^1+e_i^1}e_{i+1}^1,
\end{split}
\end{equation*}
which implies
$$d_{i,i}^{1,1}=id_{1,1}^{1,1},\quad d_{i,i+1}^{1,1}=d_{1,2}^{1,1},\quad 3\leq i\leq m_1-1.$$

From the equalities,
\begin{equation*}
\begin{split}
\Delta(e_1^1+e_2^1)&=d_{1,1}^{1,1}e_1^1+d_{1,2}^{1,1}e_2^1+d_{2,2}^{1,1}e_2^1+d_{2,3}^{1,1}e_{3}^1\\
&=D_{e_1^1+e_2^1}(e_1^1+e_2^1)=D_{e_1^1+e_2^1}(e_1^1)+D_{e_1^1+e_2^1}(e_2^1)\\
                    & =\alpha_{e_1^1+e_2^1}e_1^1+\beta_{e_1^1+e_2^1}e_2^1+2\alpha_{e_1^1+e_2^1}e_2^1+\beta_{e_1^1+e_2^1}e_{3}^1,
\end{split}
\end{equation*}
and
\begin{equation*}
\begin{split}
\Delta(e_1^1+e_{m_1}^1)&=d_{1,1}^{1,1}e_1^1+d_{1,2}^{1,1}e_2^1+d_{m_1,m_1}^{1,1}e_{m_1}^1\\
\Delta(e_1^1+e_{m_1}^1)&=D_{e_1^1+e_{m_1}^1}(e_1^1+e_{m_1}^1)=D_{e_1^1+e_{m_1}^1}(e_1^1)+D_{e_1^1+e_{m_1}^1}(e_{m_1}^1)=\\
                    & =\alpha_{e_1^1+e_{m_1}^1}e_1^1+\beta_{e_1^1+e_{m_1}^1}e_2^1+m_1\alpha_{e_1^1+e_{m_1}^1}e_{m_1}^1,
\end{split}
\end{equation*}
we get that $$d_{2,2}^{1,1}=2d_{1,1}^{1,1}, \quad d_{m_1,m_1}^{1,1}=m_1d_{1,1}^{1,1}.$$


Now for $2\leq t\leq s,$ $2\leq i\leq m_t-1,$ we consider
$$\Delta(e_i^t+e_1^t+e_1^1)=d_{i,i}^{t,t}e_i^t+d_{i,i+1}^{t,t}e_{i+1}^t+d_{1,1}^{t,t}e_1^t+d_{1,2}^{t,t}e_{2}^t+d_{1,1}^{1,1}e_1^1+
d_{1,2}^{1,1}e_2^1.$$

On the other hand,
\begin{equation*}
\begin{split}
\Delta(e_i^t+e_1^t+e_1^1)&=D_{e_i^t+e_1^t+e_1^1}(e_i^t+e_1^t+e_1^1)=((i-1)\alpha_{e_i^t+e_1^t+e_1^1}+\eta_{e_i^t+e_1^t+e_1^1,t})e_i^t+\\
                    &+\beta_{e_i^t+e_1^t+e_1^1}e_{i+1}^t+\eta_{e_i^t+e_1^t+e_1^1,t}e_{1}^t+\beta_{e_i^t+e_1^t+e_1^1}e_{2}^t+\\
                   &+\alpha_{e_i^t+e_1^t+e_1^1}e_1^1+\beta_{e_i^t+e_1^t+e_1^1}e_2^1.
\end{split}
\end{equation*}

Comparing the coefficients at the basis elements $e_{i}^t,\ e_1^t$ and $e_1^1,$ we get
$$\alpha_{e_i^t+e_1^t+e_1^1}=d_{1,1}^{1,1},\quad (i-1)\alpha_{e_i^t+e_1^t+e_1^1}+\eta_{e_i^t+e_1^t+e_1^1,t}=d_{i,i}^{t,t},\quad  \eta_{e_i^t+e_1^t+e_1^1,t}=d_{1,1}^{t,t},$$
which implies
$$d_{i,i}^{t,t}=(i-1)d_{1,1}^{1,1}+d_{1,1}^{t,t},\quad 2\leq t\leq s,\quad \leq i\leq m_t-1.$$

Similarly, from
\begin{equation*}
\begin{split}
\Delta(e_{m_t}^t+e_1^t+e_1^1)&=d_{m_t,m_t}^{t,t}e^t_{m_t}+d_{1,1}^{t,t}e_1^t+d_{1,2}^{t,t}e_{2}^t+d_{1,1}^{1,1}e_1^1+
d_{1,2}^{1,1}e_2^1 \\&=D_{e_{m_t}^t+e_1^t+e_1^1}(e_{m_t}^t+e_1^t+e_1^1)=((m_t-1)\alpha_{e_{m_t}^t+e_1^t+e_1^1}+
\eta_{e_{m_t}^t+e_1^t+e_1^1,t})e_{m_t}^t+\\
                    &+\eta_{e_{m_t}^t+e_1^t+e_1^1,t}e_{1}^t+\beta_{e_{m_t}^t+e_1^t+e_1^1}e_{2}^t+
                    \alpha_{e_{m_t}^t+e_1^t+e_1^1}e_1^1+\beta_{e_{m_t}e_i^t+e_1^t+e_1^1}e_2^1,
\end{split}
\end{equation*}
we get that
$$d_{m_t,m_t}^{t,t}=(m_t-1)d_{1,1}^{1,1}+d_{1,1}^{t,t},\quad 2\leq t\leq s.$$

Now, we consider
$$\Delta(x_1+e_2^1)=b_{1,1}^1e_1^1+d_{2,2}^{1,1}e_2^1+d_{2,3}^{1,1}e_3^1.$$

On the other hand,
\begin{equation*}
\begin{split}
\Delta(x_1+e_2^1)&=D_{x_1+e_2^1}(x_1+e_2^1)=-\beta_{x_1+e_2^1}e_1^1+2\alpha_{x_1+e_2^1}e_2^1+\beta_{x_1+e_2^1}e_3^1.
\end{split}
\end{equation*}

Comparing the coefficients at the basis elements $e_3^1$ and $e_1^1,$ we get
$\beta_{x_1+e_2^1}=d_{2,3}^{1,1},$ $-\beta_{x_1+e_2^1}=b_{1,1}^1,$
which implies
$$b_{1,1}^1=-d_{2,3}^{1,1}=-d_{1,2}^{1,1}.$$

Thus, we obtain that the local derivation $\Delta$ has the following form:
$$\begin{array}{ll}
\Delta(e^1_i)=id_{1,1}^{1,1}e^1_i+d_{1,2}^{1,1}e^1_{i+1},&1 \leq i \leq m_1-1,\\[2mm]
\Delta(e^1_{m_1})=m_1d_{1,1}^{1,1}e^1_{m_1},&\\[2mm]
\Delta(e^t_i)=((i-1)d_{1,1}^{1,1}+d_{1,1}^{t,t})e^t_i+d_{1,2}^{1,1}e^t_{i+1},&  2\leq t\leq s ,\quad 1 \leq i \leq m_t-1,\\[1mm]
\Delta(e^t_{m_t})=((m_t-1)d_{1,1}^{1,1}+d_{1,1}^{t,t})e^t_{m_t},&  2\leq t\leq s,\\[1mm]
\Delta(x_1)=-d_{1,2}^{1,1}e^1_1.
 \end{array}$$

Proposition \ref{pro} implies that $\Delta$ is a derivation.
Hence, every local derivation on $R(N_{m_1,\dots, m_s},s)$ is a derivation.

\end{proof}

\subsection{Local derivation of solvable Leibniz algebras with abelian nilradical}

Now we shall give the main result concerning local derivations on solvable  Leibniz algebras with abelian nilradicals.
\begin{theorem} 
Any local derivation on the algebra $\mathcal{L}_t$ is a derivation.

\end{theorem}
\begin{proof}  For any local derivation $\Delta$ on the algebra ${\mathcal L_t},$ we put the derivation $D,$ such that:
\[ \quad D(f_j)=a_jf_j, \quad D(x_j)=\alpha_jb_jf_j, \quad 1\leq j\leq n,\]

Then, we get  $$\Delta(f_j)=D_{f_j}(f_j)=a_jf_j,\quad \Delta(x_j)=D_{x_j}(x_j)=\alpha_jb_jf_j.$$

Hence, $\Delta$ is a derivation.
\end{proof}

In the following theorem, we show that $(n+1)$-dimensional solvable Leibniz algebras with $n$-dimensional abelian nilradical have a local derivation which is not a derivation.

\begin{theorem}  $(n+1)$-dimensional solvable Leibniz algebras $R_1$ and $R_2$ (see Theorem \ref{thmQ=1}), admit a local derivation which is not a derivation.
\end{theorem}

\begin{proof} Let us consider the linear operator $\Delta$ on $R_1$ and $R_2$, such that
$$\Delta\left(\sum\limits_{i=1}^n\xi_ie_i+\xi_{n+1}x\right)=2\xi_1e_{n-1}+\xi_2e_n.$$

By Proposition \ref{123}, it is not difficult to see that $\Delta$ is not a derivation.
We show that, $\Delta$ is a local derivation on $R_1$ and $R_2$.

Consider the derivations $D_1$ and $D_2$ on the algebras $R_1$ and $R_2,$ defined as:
\begin{equation*}\begin{split}
D_1\left(\sum\limits_{i=1}^n\xi_ie_i+\xi_{n+1}x\right)&=\xi_1e_{n-1}+\xi_2e_n,\\
D_2\left(\sum\limits_{i=1}^n\xi_ie_i+\xi_{n+1}x\right)&=\xi_1e_n.
\end{split}
\end{equation*}

Now, for any $\xi=\sum\limits_{i=1}^{n}\xi_ie_i+\xi_{n+1}x,$ we shall find a derivation $D,$ such that $\Delta(\xi) = D(\xi).$

If  $\xi_1=0,$ then
$$\Delta(\xi)=0=D_2(\xi).$$

If $\xi_1\neq 0,$ then setting
 $D=2D_1+tD_2,$ where $t=-\frac{\xi_2}{\xi_1},$ we obtain that
\begin{equation*}\begin{split}
 \Delta(\xi)&=2\xi_1e_{n-1}+\xi_2e_n=2\xi_1e_{n-1}+(2\xi_2+t\xi_1)e_n=2(\xi_1e_{n-1}+\xi_2e_n)+t\xi_1e_n=\\
                            &= 2D_1(\xi)+tD_2(\xi)=D(\xi).
\end{split}\end{equation*}

Hence, $\Delta$ is a local derivation.
\end{proof}

\section{$2$-local derivation of solvable Leibniz algebras}
\subsection{$2$-local derivation of solvable  Leibniz algebra $R(N_{m_1,\dots, m_s},s)$}

Now we shall give the main result concerning of the $2$-local derivations of solvable Leibniz algebra $R(N_{m_1,\dots, m_s},s).$

Consider an element $q=\sum\limits_{t=1}^s\sum\limits_{i=1}^{m_t}e_i^t$ of $R(N_{m_1,\dots, m_s},s).$

\begin{lemma}\label{aa}Let $\nabla$ be a 2-local derivation of $R(N_{m_1,\dots, m_s},s)$, such that
$\nabla(q)=0.$ Then $\nabla\equiv0.$
\end{lemma}

\begin{proof} Take an element $a_{q,\xi}\in R,$ such that
$$\nabla(q)=[a_{q,\xi},q],\ \ \ \nabla(\xi)=[a_{q,\xi},\xi].$$
Then
\begin{equation*}\begin{split}
0=\nabla(q)&=[q,a_{q,\xi}]=\left[\sum\limits_{t=1}^s\sum\limits_{i=1}^{m_t}e_i^t,
\sum\limits_{i=1}^nc_ix_i+\sum\limits_{t=1}^s\sum\limits_{i=1}^{m_t}d_i^te_i^t
\right]\\
                &=\sum\limits_{i=1}^{m_1}c_1ie_1^1+\sum_{t=2}^s\sum\limits_{i=1}^{m_t}c_1(i-1)e_{i}^t+\sum_{t=2}^s\sum\limits_{i=1}^{m_{t}}c_t(i-1)e_{i}^t+               \sum_{t=1}^s\sum\limits_{i=1}^{m_t-1}d_1^1e_{i+1}^t,
\end{split}
\end{equation*}
which implies, $c_i=d_1^1=0$ for all $1\leq i\leq n.$

Thus, $a_{q,\xi}=\sum\limits_{t=2}^sd_1^te_1^t+
\sum\limits_{i=2}^{m_1}d_i^1e_i^1+\sum\limits_{t=2}^s\sum\limits_{i=2}^{m_t}d_i^te_i^t.$

Consequently, for any element $\xi \in R(N_{m_1,\dots, m_s},s),$ we have
\begin{equation*}\begin{split}
\nabla(\xi)=&[\xi,a_{q,\xi}]=[\sum\limits_{i=1}^n\lambda_ix_i+\sum\limits_{t=1}^s\sum\limits_{i=1}^{m_t}\mu_i^te_i^t,
\sum\limits_{t=2}^sd_1^te_1^t+\sum\limits_{i=2}^{m_1}d_i^1e_i^1+\sum\limits_{t=2}^s\sum\limits_{i=2}^{m_t}d_i^te_i^t ]=0.
\end{split}
\end{equation*}

\end{proof}

\begin{theorem}\label{abc} Any $2$-local derivation
of the solvable Leibniz algebra $R(N_{m_1,\dots, m_s},s)$ is a derivation.
\end{theorem}

\begin{proof}  Let $\nabla$ be a 2-local derivation of $R(N_{m_1,\dots, m_s},s).$  Take a derivation $D_{\xi,q}$ such that
\begin{equation*}
\nabla(q)=D_{\xi,q}(q).
\end{equation*}

Set $\nabla_1=\nabla-D_{\xi,q}.$ Then $\nabla_1$ is a 2-local
derivation, such that $\nabla_1(q)=0.$

By Lemma \ref{aa}, we get that $\nabla_1\equiv0.$  Hence, $\nabla=D_{\xi,q},$ i.e., $\nabla$ is a
derivation.
\end{proof}

\subsection{$2$-local derivation of solvable  Leibniz algebras with alebian nilradical}
Now we shall give the result concerning of $2$-local derivations of solvable  Leibniz algebras with abelian nilradical.

\begin{theorem}\label{abcd}
The algebra $\mathcal{L}_t$ admits a $2$-local derivation which is not a derivation.
\end{theorem}
\begin{proof}  Let us define a homogeneous non additive function $f$ on $\mathbb{C}^{2}$ as follows

\[ f(z_{1},z_{2}) = \begin{cases}
\frac{z^{2}_{1}}{z_{2}}, & \text{if $z_{2}\neq0$,}\\
0, & \text{if $z_{2}=0$},
\end{cases} \]
where $(z_1,z_2)\in\mathbb{C}^{2}.$

 Define the operator $\nabla$ on $\mathcal{L}_t,$ such that
\begin{equation}\label{fff}
\nabla(\xi)=f(\xi_{1},\xi_{n+1})f_{1},
\end{equation}
for any element $\xi=\sum\limits_{i=1}^{n}\xi_{i}f_{i}+\sum\limits_{i=1}^{n}\xi_{n+i}x_i,$

The operator $\nabla$ is not a derivation, since it is not
linear.

Let us show that, $\nabla$ is a $2$-local derivation. For this purpose, define a derivation $D$ on $\mathcal{L}_t$ by

$$D(\xi)
=(a\xi_{1}+b\xi_{n+1})f_1.$$

For each pair of elements $\xi$ and $\eta,$ we choose $a$ and $b,$ such that
$\nabla(\xi)=D(\xi)$ and $\nabla(\eta)=D(\eta)$.
Let us rewrite the above equalities as system of linear equations with respect to unknowns $a,\ b$  as follows

\begin{equation}\label{sss}
\begin{cases}
\xi_{1} a+\xi_{n+1}b =f(\xi_{1}, \xi_{n+1}), \\
\eta_{1} a+\eta_{n+1} b = f(\eta_{1}, \eta_{n+1}).
\end{cases}
\end{equation}

\textbf{Case 1}. $\xi_{1}\eta_{n+1}-\xi_{n+1}\eta_{1}=0$. In this case, since the right-hand side of the system (\ref{sss}) is homogeneous, it has infinitely many solutions.

\textbf{Case 2}. $\xi_{1}\eta_{n+1}-\xi_{n+1}\eta_{1}\neq0$. In this case, the system (\ref{sss})  has a unique solution.

\end{proof}

\begin{proposition}
Any $2$-local derivation of the algebra $R_1$ is a derivation.
\end{proposition}

\begin{proof} Let $\nabla$ be a $2$-local derivation on $R_1,$ such that $\nabla(f_1)=0.$
Then for any element $\xi=\sum\limits_{i=1}^{n}\xi_if_i+\xi_{n+1}x\in R_1,$ there exists a derivation $D_{f_1,\xi}(\xi)$, such that
$$\nabla(f_{1})=D_{f_1,\xi}(f_1),\quad \nabla(\xi)=D_{f_1,\xi}(\xi).$$

Hence,
$$0=\nabla(f_{1})=D_{f_1,\xi}(f_1)=\sum\limits_{i=1}^n\alpha_if_i,$$
which implies,  $\alpha_{i}=0$ for $1\leq i\leq n.$

Consequently, from the description of the derivation $R_1,$ we conclude that
$D_{f_1,\xi}=0.$
Thus, we obtain that if $\nabla(f_1)=0,$ then
$
\nabla\equiv0.
$

Let now $\nabla$ be an arbitrary $2$-local derivation of \(R_1\).
Take a derivation $D_{f_1,\xi},$ such that
\begin{equation*}
\nabla(f_1)=D_{f_1,\xi}(f_1)\ \ \text{and} \ \ \nabla(\xi)=D_{f_1,\xi}(\xi).
\end{equation*}

Set $\nabla_1=\nabla-D_{f_1,\xi}.$ Then $\nabla_1$ is a $2$-local
derivation, such that $\nabla_1(f_1)=0.$ Hence $\nabla_1(\xi)=0$ for all $\xi\in R_1,$  which implies $\nabla=D_{f_1,\xi}.$
Therefore, $\nabla$ is a
derivation.
\end{proof}

\begin{theorem}
Solvable Leibniz algebra $R_2$ admits a $2$-local derivation which is not a derivation.
\end{theorem}

\begin{proof} The proof is similar to the proof of Theorem \ref{abcd}.
\end{proof}

\end{document}